\documentclass[11pt]{amsart}

\usepackage{amsmath,amssymb,amscd,amsfonts,verbatim}
\usepackage{enumerate}
\usepackage[mathscr]{eucal}
\usepackage{color}

\usepackage[pdftex]{hyperref}
\hypersetup{citecolor=blue,linktocpage}

\usepackage[colorinlistoftodos]{todonotes}

\newtheorem{thm}{Theorem}[section]
\newtheorem{lem}[thm]{Lemma}

\newtheorem{prop}[thm]{Proposition}
\newtheorem{cor}[thm]{Corollary}

\newtheorem{assu-nota}[thm]{Assumption--Notation}

\theoremstyle{definition}

\newtheorem{rem}[thm]{Remark}
\newtheorem{ex}[thm]{Example}
\newtheorem{qst}[thm]{Question}

\newcommand{\inv}{^{-1}}
\newcommand{\unu}{^{\nu}}

\newcommand{\C}{\mathbb C}
\newcommand{\Z}{\mathbb Z}

\newcommand{\pp}{\mathbb P}
\newcommand{\bC}{\mathbb C}
\newcommand{\bG}{\mathbb G}
\newcommand{\bP}{\mathbb P}

\newcommand{\OO}{\mathcal O}
\newcommand{\cO}{\mathcal O}

\DeclareMathOperator{\Aut}{Aut}
\DeclareMathOperator{\Pic}{Pic}
\DeclareMathOperator{\NS}{NS}

\DeclareMathOperator{\Alb}{Alb}
\DeclareMathOperator{\albdim}{albdim}
\DeclareMathOperator{\Hom}{Hom}

\DeclareMathOperator{\Spec}{Spec}
\DeclareMathOperator{\im}{Im}
\DeclareMathOperator{\rk}{rk}

\DeclareMathOperator{\pr}{pr}

\DeclareMathOperator{\End}{End}
\DeclareMathOperator{\rR}{R}
\DeclareMathOperator{\coker}{coker}
\DeclareMathOperator{\res}{res}

\newcommand{\al}{\alpha}
\newcommand{\be}{\beta}
\newcommand{\ga}{\gamma}

\newcommand{\Ga}{\Gamma}

\newcommand{\si}{\sigma}

\newcommand{\cP}{\mathcal{P}}

\newcommand{\lra}{\longrightarrow}

\numberwithin{equation}{section}

\title{Brill-Noether loci for divisors    on irregular varieties}
\author{Margarida Mendes Lopes, Rita Pardini and Gian Pietro Pirola}

\thanks{
The first author is a member of the Center for Mathematical
Analysis, Geometry and Dynamical Systems (IST/UTL).  The second and the third author are members of G.N.S.A.G.A.--I.N.d.A.M. This research was partially supported by FCT (Portugal) through program POCTI/FEDER and
Project PTDC/MAT/099275/2008 and by MIUR (Italy) through  PRIN 2008 ``Geometria delle variet\`a algebriche e dei loro spazi di moduli" and PRIN 2009 ``Spazi di moduli e Teoria di Lie"}

\begin{document}
\begin{abstract}
We take up the study  of the Brill-Noether  loci $W^r(L,X):=\{\eta\in \Pic^0(X)\ |\ h^0(L\otimes\eta)\ge r+1\}$, where $X$ is a smooth projective variety of dimension $>1$, $L\in \Pic(X)$,  and $r\ge 0$ is an integer.

By studying  the infinitesimal structure of these loci and  the Petri map (defined  in analogy with the case of curves),   we obtain  lower bounds for $h^0(K_D)$, where    $D$  is a divisor that moves linearly  on a smooth projective variety $X$ of maximal Albanese dimension. In this way  we sharpen  the results of \cite{xiaoirreg} and we generalize them  to dimension $>2$.

In the $2$-dimensional case we prove an existence theorem:   we define a Brill-Noether number $\rho(C, r)$ for   a curve $C$ on  a  smooth surface $X$ of maximal Albanese dimension and we prove, under some mild additional assumptions,  that if $\rho(C, r)\ge 0$ then $W^r(C,X)$  is nonempty of dimension $\ge \rho(C,r)$. 

Inequalities for the numerical invariants of curves that do not move linearly on a surface of maximal Albanese dimension are obtained as an application of the previous results.

\end{abstract}
\maketitle
\tableofcontents

\section{Introduction}

The classical Brill-Noether  theory studies  the  loci $$W^r_d(C):=\{ L\in \Pic(C) |\deg L= d, \  h^0(L)\ge r+1\},$$ where $C$ is a smooth projective curve of genus $g\ge 2$. We refer the reader to  \cite{acgh}   for a comprehensive treatment of this beautiful topic and to \cite{acg} for further information.  We only recall here that all the theory revolves around the Brill-Noether number $\rho(g,r,d)=g-(r+1)(r+g-d)$: if $\rho(g,r,d)\ge 0$ then $W^r_d(C)$ is not empty, and if $\rho(g,r,d)> 0$ then $W^r_d(C)$ is connected of dimension $\ge \rho(g,r,d)$. In addition, if $C$ is general in moduli then  $\dim W^r_d(C)= \rho(g,r,d)$. 

Several possible generalizations of this theory have been investigated in the past years, the most studied  being the case in which divisors of fixed degree are  replaced by stable vector bundles of fixed rank and degree on the smooth curve $C$, see \cite{teixidor} for a recent survey. Generalizations to some varieties of dimension $>1$ have been considered by several people (see for instance  \cite{dl}, \cite{h}, \cite{le}).

 Moreover,  Brill-Noether type loci for higher dimensional varieties occur naturally   in the theory of deformations, Hilbert schemes,  Picard schemes and  Fourier-Mukai transforms, but  usually not as the main object of study.  
In \cite{costa-miro} the case of vector bundles on an arbitrary  smooth projective variety is considered under the assumption that all the cohomology groups of degree $>1$ vanish.
In \cite{kleiman-BN}  the Brill-Noether loci are defined in great generality for relative subschemes of any codimension of a family of projective schemes, but the theory is developed only in the case of linear series on smooth projective curves.  Any concrete theory of special divisors, like for instance existence theorems, seems impossible in such generality.

Here we take up what seems to us the most straightforward generalization of the  classical theory of linear series on curves, namely the case of line bundles on an arbitrary projective variety.   In this setup, the Brill-Noether loci are  a special instance  of the  cohomological support loci, whose study has been started in \cite{GL1}, \cite{GL2}, focusing  on the case of topologically trivial line bundles, 
and  has been extended and refined  in the context of rank one local systems (see for instance \cite{dps}).  However, our point of view and that of \cite{GL1}, \cite{GL2} are different, since we look for lower bounds on the dimension of these loci rather than for upper bounds.

Let's now summarize the  content of  the paper. Given a projective variety $X$,  a line bundle $L\in \Pic(X)$ and an integer $r\ge 0$, we set $W^r(L,X):=\{\eta\in \Pic^0(X)\ |\ h^0(L\otimes \eta)\ge r+1\}$.  First we  recall  the  natural scheme structure on $W^r(L,X)$ and, by analysing it, we    show that,  if $X$ has maximal Albanese dimension (i.e,  it has generically finite Albanese map)
and   $D$ is an effective  divisor contained in the fixed part of $|K_X|$,   then $0\in W(D,X)$ is an isolated point (Corollary \ref{cor:fixed-rigid}).

Then we focus on  two special cases: (a)  singular projective curves  and (b) smooth surfaces of maximal Albanese dimension. In case (a) we assume that $X$ is    a connected reduced projective curve and,   since  in general  $\Pic^0(X)$ is not  an abelian variety, 
we study the intersection of $W^r(L,X)$ with a compact subgroup $T\subseteq \Pic^0(X)$ of dimension $t$. 
We  define the  Brill-Noether number $\rho(t, r, d):=t-(r+1)(p_a(C)-d+r)$, where $d:=\deg L$. In this set-up  we prove the exact analogue of the existence  theorem  of Brill-Noether theory under a technical assumption on $T$ (Theorem \ref{thm:BN}, cf.  also Remark \ref{rem:ipotesi}). 
We do not consider here the very important theory of the compactifications of the Brill-Noether loci of singular curves and the theory of limit linear series as treated by  many authors (see, for instance,  {\cite{G, EH,  A, EK, C, acg}, and  see also  \cite {acg} for a complete bibliography). Our reason is that    we are interested in line bundles coming  from a smooth complete variety via restriction, and these are naturally  parametrized by a compact subgroup. 

Theorem \ref{thm:BN} is a key step in our approach to  case (b): we combine it with the generic vanishing theorem of Green  and Lazarsfeld to obtain 
 an  analogue of the existence theorem of Brill-Noether theory. Given a curve $C$ on a surface $S$ with irregularity $q>1$,  we consider the image $T$ of  the natural map $\Pic^0(S)\to \Pic^0(C)$. If this map has finite kernel  and we take $L=\OO_C(C)$, then   the Brill-Noether number  introduced  above can be written as 
 $\rho(C,r):=q-(r+1)(p_a(C)-C^2+r)$.  (For $r=0$ we write simply  $\rho(C)$).

For surfaces without irrational pencils of genus $>1$  and reduced curves $C$ all of whose   components have  positive self-intersection,   we prove that  if    $\rho(C,r)>1$,  then $W^r(C,S)$ is nonempty  of dimension $\ge\min\{q,  \rho(C, r)\}$,  and that for   $\rho(C,r)=1$ the same statement holds under an additional  assumption. In the specific case $r=0$ and under the same hypotheses we are also able to show that $C$ actually moves algebraically in a family of  dimension  $\ge \min\{q,  \rho(C\})$. For the precise statement  see Theorem  \ref{thm:Cmoves}. We remark that the assumptions on the Albanese map and on the structure of $V^1$ in these results are quite mild (see  Remark \ref{rem:mild} and  Section \ref{sec:Alb}).
 
\smallskip

A  second theme of the paper, strictly interwoven with the analysis of the Brill-Noether loci, is the study of the restriction map $r_D\colon H^0(K_X)\to H^0(K_X|_D)$, where $X$ is a smooth variety of maximal Albanese dimension and $D\subset X$ is an effective divisor whose image via the Albanese map $a\colon X\to \Alb(X)$ generates $\Alb(X)$. In Proposition \ref{prop:xiao} we establish a uniform lower bound for the rank of $r_D$  under the only assumption that $D$ is not contained in the ramification locus of the Albanese map (this is also one of the ingredients of the proof of the  Brill-Noether type  result for surfaces). 
Then we show how one can improve on this bound if the tangent space to $W^0(D, X)$ at $0$ has positive dimension (Proposition \ref{prop:xiao-petri}); in order to do this we introduce and study, in analogy with  the case of curves, the Petri map $H^0(D)\otimes H^0(K_X-D)\to H^0(K_X)$.  If $h^0(D)>1$, a lower bound for the rank of $r_D$ gives immediately a lower bound for $h^0(K_D)$ (Corollaries \ref{cor:genere} and \ref{cor:genere2}): in this way  we extend to arbitrary dimension the main result of \cite{xiaoirreg}, which treats the case in which $X$ is a surface and $D$ is a general fiber of a fibration $X\to \pp^1$.

All the previous results are applied in \S \ref{sec:applications} to the study of curves on a  surface of general type $S$  with $q(S):=h^0(\Omega^1_S)>1$  that   is  not fibered onto a curve of genus $>1$. More precisely, we give  inequalities for the numerical invariants of a  curve $C=\sum_iC_i$ of $S$  such that  $C_i^2>0$ for all $i$ and $h^0(C)=1$ (Corollary \ref{cor:ineq3}); in the special case in which $p_a(C)\le 2q(S)-2$  we obtain a stronger inequality  and a lower bound on the codimension of $W^0(C,S)$ in  $\Pic^0(S)$ (Corollary \ref{cor:q+1}). Remark that, apart from  the case when $p_a(C)=q(S)$ (classified in \cite{symm}), the question of the existence on a surface of general type $S$ of curves $C$ with $C^2>0$ and $p_a(C)\le 2q(S)-2$ is, as far as we know, completely open. Finally, we prove a result (Proposition \ref{prop:para}) that relates the fixed locus of the paracanonical system of $S$ to the ramification divisor of the Albanese map. 

In \S \ref{sec:examples} we collect several examples in order to illustrate the phenomena that occur for  Brill-Noether loci on surfaces and to  clarify to what extent the results that we obtain are optimal.  We also  pose some questions: in our opinion,  the most important of these  is whether a statement analogous to Theorem \ref{thm:Cmoves} holds if one replaces the effective divisor  $C$ by, say, an ample line bundle,  and whether a similar statement holds in arbitrary dimension.  The main difficulty here is that,  while in the case of curves  the cohomology of a family of  line bundles  of fixed degree is computed by a complex with only 2 terms, in the case of a variety of dimension $n$   one has to deal with a complex of length  $n+1$. 
 Hence the negativity of Picard sheaves, which  has been established for projective varieties of any dimension (cf.  \cite{lazarsfeld}, 6.3.C and 7.2.15), does not suffice alone to prove non emptyness results for the Brill-Noether loci.

In the surface case the method of restriction to curves and the use of the generic vanishing theorems overcome the cohomological problem. We are aware however that usually curves lying on surfaces are not general in the sense of the Brill-Noether theory, hence, although the existence theorem \ref{thm:Cmoves} is sharp, one cannot expect  that the Brill-Noether number computes precisely the dimension of the Brill-Noether locus in most cases. In fact, in view of the complexity of the geometry and  of the topology   of irregular surfaces (even the geographical problem has not been solved yet, cf. \cite{survey}), it is somewhat  surprising that  a single numerical invariant, such as the Brill-Noether number, can give a definite existence result for  continuous families of effective divisors on surfaces. 
Our methods  are also  very useful to attack  problems in  classification theory and questions about  on curves on surfaces, as we illustrated in \cite{symm} and in section \ref{sec:applications} of this paper.

\smallskip
In addition,  the use of the generic vanishing combined with the infinitesimal analysis in sections \ref{sec:BNdef} and  \ref{sec:restriction} shows  the importance  of the Petri map in the higher dimensional case. 

We are convinced  that the methods of the present paper together with the use of some fine obstructions theory as in \cite{para}  will give some striking new results on the theory of continuous families  of divisor on irregular varieties, which is ultimately  the Brill Noether theory.
\bigskip

\noindent {\em Acknowledgments:\/}  We wish to thank  Angelo Vistoli for  many useful mathematical discussions.

\bigskip

{\bf Notation and conventions.}  We work over the complex numbers. All varieties  are assumed to be complete. We do not distinguish between divisors on smooth varieties and the corresponding line bundles and we denote linear equivalence by $\equiv$. 
\smallskip

 Let $X$ be  a smooth projective  variety.  We denote as usual by $\chi(X)$ the  Euler characteristic of $\OO_X$, by $p_g(X)$ the {\em geometric genus}  $h^0(X, K_X)$ and by $q(X)$  the {\em irregularity} $h^0(X, \Omega^1_X)$. We denote by $\albdim(X)$ the dimension of  the image of  the Albanese map $a\colon X\to\Alb(X)$. As usual, a {\em fibration of $X$} is a surjective morphism with connected fibers $X\to Y$, where $Y$ is a variety with $\dim Y<\dim X$.  We say that $X$ has  
    an {\em irrational pencil of genus $g>0$} if it admits a fibration  $X\to B$  onto a smooth  curve of  genus $g>0$.
    
    If $D$ is an effective divisor of a smooth variety $X$ we denote by $p_a(D)$ the {\em  arithmetic genus}  $\chi(K_D)-1$, where  $K_D$ is the canonical divisor of $D$. In particular, if $\dim X=2$ and $D$ is a nonzero effective divisor  (a {\em curve}) then by the adjunction formula the arithmetic genus of $D$ of $S$ is  $p_a(D)=(K_SD + D^2)/2 + 1$;  the curve $D$ is said to be {\em $m$-connected\/} if, given any decomposition
$D = A+ B$ of $D$ with  $A, B>0$, one has $AB \geq m$.
\smallskip

Given a product of varieties $V_1\times...\times V_n$ we  denote by ${\pr_i}$ the projection onto the $i-$th factor. 

\section{Preliminaries on irregular varieties} \label{sec:Alb} 

We recall some  by now classical  results on irregular varieties that are used repeatedly throughout the paper.
\subsection{Albanese dimension and irregular fibrations}\label{ssec:alb}

Let $X$ be a smooth projective variety of dimension $n$. 
The {\em Albanese dimension} $\albdim(X)$ is defined as the dimension of the image of the Albanese map of $X$; in particular,  $X$ has {\em maximal Albanese dimension} if its Albanese map is generically finite onto its image and it is  of {\em Albanese general type} if in addition $q(X)>n$. For a normal  variety $Y$, we define the Albanese variety $\Alb(Y)$ and all the related notions  by considering  any smooth projective model of $Y$.

An {\em irregular fibration} $f\colon X\to Y$ is a morphism with positive dimensional connected fibers onto a normal  variety $Y$ with $\albdim Y=\dim Y>0$; the map $f$ is called an  {\em Albanese general type  fibration} if in addition $Y$ is of Albanese general type. If $\dim Y=1$, then $Y$ is a smooth curve of genus $b>0$; in that case, $f$ is called an {\em irrational pencil of genus $b$} and it is an Albanese general type fibration if and only if $b>1$. 

Notice that if $q(X)\ge n$ and $X$ has no Albanese general type fibration, then $X$ has  maximal Albanese dimension.

The so-called generalized Castelnuovo--de Franchis Theorem (see  \cite[Thm. 1.14]{catanese-irregular} and Ran \cite{ran}) shows  how  the existence of Albanese general type  fibrations  is detected by  the cohomology of $X$:
\begin{thm}[Catanese, Ran]\label{thm:cast-cat} 
The smooth projective variety $X$ has an Albanese general type  fibration $f\colon X\to Y$ with $\dim Y\le k$ if and only if there exist independent $1$-forms $\omega_0,\dots \omega_k \in H^0(\Omega^1_X)$ such that $\omega_0\wedge \omega_1\wedge \dots\wedge  \omega_k =0\in H^0(\Omega^{k+1}_X)$.
\end{thm}

So in particular the existence of irrational pencils of genus $>1$ is equivalent to the existence of two independent $1$-forms $\alpha, \beta\in H^0(\Omega^1_X)$ such that $\alpha\wedge \beta=0$.

\subsection{Generic vanishing}\label{ssec:GV}
Let $X$ be a projective variety of dimension $n$ and let $L\in \Pic(X)$; 
the  generic vanishing loci, or Green-Lazarsfeld loci,   are defined as   $V^i( X):=\{\eta\ | \ h^i(\eta)>0\}\subseteq \Pic^0(X)$, $i=0,\dots n$.  They have been object of intense study since the groundbreaking papers \cite{GL1}, \cite{GL2} and their structure is very well understood.

We only summarize here for later use the properties  of $V^1(X)$, established in \cite{GL1}, \cite{GL2}, \cite{beauville-annulation}, \cite{beauville-errata} and \cite{Si}:
\begin{thm}\label{thm:GV1}
Let $X$ be a smooth projective variety; then:
\begin{enumerate} 
\item if $X$ has maximal Albanese dimension, then $V^1(X)$ is a proper closed subset of $\Pic^0(X)$ whose components are translates by torsion points of abelian subvarieties;
\item if $X$ has no irrational pencil of genus $>1$, then $\dim V^1(X)\le 1$ and  $0\in V^1(X)$ is an isolated point.
\end{enumerate}
\end{thm}

\section{The Brill-Noether loci} \label{sec:BNdef} 

In this section we recall the definition of the Brill-Noether loci and some general facts on their geometry. 
The scheme structure and the tangent space to the  Brill-Noether loci have been described  in several contexts, however  for clarity's sake we  choose to spell out and prove the properties we need in the sequel.   We close the section by proving some properties of the ramification divisor of the Albanese map and of the fixed divisor of the canonical system of a variety of maximal Albanese dimension (Proposition \ref{prop:dv} and Corollary  \ref{cor:fixed-rigid}).
\smallskip

Let $X$ be a projective variety   and let $L\in \Pic(X)$. For $r\ge 0$ we define the {\em Brill-Noether locus} $$W^r(L,X):=\{\eta\in \Pic^0(X)\ |\ h^0(L\otimes\eta)\ge r+1\}.$$
 If $T\subseteq \Pic^0(X)$ is a subgroup, we set $W^r_T(L,X):=W^r(L,X)\cap T$. For $r=0$ we write $W(L,X)$ instead of $W^0(L,X)$. 

\begin{rem}
When $X$ is a smooth curve, the  Brill-Noether loci are a very  classical object of study (cf. \cite{acgh}, Ch. III, IV and  V). The definition we give here is slightly different from the classical one, which consists in fixing a class $\lambda\in \NS(X)$  and defining the Brill-Noether locus as $W^r_{\lambda}(X):=\{M\in \Pic^{\lambda}(X)\ |\ h^0(M)\ge r+1\}$, where $\Pic^{\lambda}(X)$ denotes the preimage of $\lambda$ via the natural map $\Pic(X)\to \NS(X)$. Of course,  if $\lambda$ is the class of $L$ in $\NS(X)$, then  $W^r(L,X)$ is mapped isomorphically onto $W^r_{\lambda}(X)$ by the translation by $L\in \Pic(X)$. Our choice of definition is motivated by technical reasons that become apparent, for instance, in the proof of Theorem \ref{thm:Cmoves}.
\end{rem}

By the semicontinuity theorem (cf. \cite{mumford-abelian},  p. 50) the Brill-Noether loci are closed in $\Pic^0(X)$.  In fact they are a particular  case of the cohomological support loci introduced in \cite[\S 1]{GL1}. 

 The scheme structure of $W^r(L,X)$ is described by  following  the approach of \cite{kleiman-BN}.  Our point of view differs slightly from \cite{kleiman-BN}  in that we consider line bundles rather than subschemes. 
  
We recall the following consequence of  Grothendieck duality: 
\begin{lem}\label{lem:grothendieck}
 Let $X$ be a projective variety of dimension  $n$, let $L\in \Pic(X)$  and let $\cP$ be  a  Poincar\'e line bundle on $X\times \Pic^0(X)$. Then there exists a coherent sheaf $Q$  on $\Pic^0(X)$, unique up to canonical isomorphism,  such that:
\begin{enumerate}
\item  for every coherent sheaf $M$ on $\Pic^0(X)$ there is a canonical isomorphism $\underline{\Hom}_{\OO_{\Pic^0(X)}}(Q, M)\cong {\pr_2}_*(\cP\otimes \pr_1^*L \otimes \pr_2^*M)$; \item if $X$ is Gorenstein, then $Q\cong \rR\!^n{\pr_2}_*(\pr_1^*(K_X-L)\otimes \cP^{\vee})$.
\end{enumerate}
\end{lem}
\begin{proof} (i) Follows by applying  \cite[Thm.7.7.6] {ega32} to the morphism $\pr_2\colon X\times \Pic^0(X)\to \Pic^0(X)$ and to the sheaf $\cP\otimes \pr_1^*L$.
\medskip 

(ii) By (i) it is enough to show that for every coherent sheaf $M$ on $\Pic^0(X)$ there is a canonical isomorphism $\underline{\Hom}_{\OO_{\Pic^0(X)}}( \rR\!^n{\pr_2}_*(\pr_1^*(K_X-L)\otimes\cP^{\vee}, M)\cong {\pr_2}_*(\cP\otimes \pr_1^*L \otimes \pr_2^*M)$.  If $X$ is Gorenstein, then $\pr_1^*\!\omega_X$ is the relative dualizing sheaf for $\pr_2\colon X\times \Pic^0(X)\to \Pic^0(X)$ and, since $X$ is Cohen-Macaulay,  the required functorial isomorphism exists  by \cite[Thm. 21]{kleiman_duality}.
\end{proof}

By Lemma \ref{lem:grothendieck} (i), a point $\eta\in \Pic^0(X)$ belongs to $W^r(L,X)$ iff   $\dim_{\C}(Q\otimes \C(\eta))\ge r+1$; hence we give $W^r(L,X)$ 
 the $r$-th Fitting subscheme structure associated with the sheaf $Q$. Notice that, since $\mathcal P$ is determined up to tensoring with $\pr_2^*M$ for $M$ a line bundle on $\Pic^0(X)$, $Q$ is also determined up to tensoring with $M$; however $Q$ and $Q\otimes M$ have the same Fitting subschemes, hence our definition is well posed.

Given $\eta \in \Pic^0(X)$, we identify as usual the tangent space to $\Pic^0(X)$ at the point $\eta$ with $H^1(\OO_X)$; then, generalizing the case when $X$ is a curve, we have the following description of the  Zariski tangent space to $W^r(L,X)$. \begin{prop}\label{prop:BNtangent}
 Let $r\ge 0$ be an integer, let $X$ be a  projective variety, let $L\in \Pic(X)$ and let $\eta\in W^r(L,X)$. Then:
\begin{enumerate}
\item if $\eta\in W^{r+1}(L,X)$, then $T_{\eta}W^r(L, X)=H^1(\OO_X)$;

\item  if $\eta\not \in W^{r+1}(L,X)$, then $T_{\eta}W^r(L, X)$ is the kernel  of the linear map $H^1(\OO_X)\to \Hom(H^0(X, L+\eta), H^1(X, L+\eta))$ induced by cup product.
\end{enumerate}
\end{prop}
\begin{proof} Let $Q$ be the coherent sheaf of Lemma \ref{lem:grothendieck}. As usual, we denote by   $\C[\epsilon]$  the algebra of dual numbers. We regard an element $v\in H^1(\OO_X)$ as a morphism $v\colon \Spec \C[\epsilon]\to \Pic^0(X)$ mapping the closed point of $\Spec\C[\epsilon]$  to $\eta$  and we denote by $Q_v$ the pull back of $Q$ via $v$. 
By the functorial properties  of Fitting ideals, $v$ is in the tangent space to $W^r(L,X)$ iff the $r$-th Fitting ideal of $Q_v$ as a $\C[\epsilon]$-module  vanishes.  Set  $m:=h^0(X,L+\eta)$.  By the definition of $Q$ (Lemma \ref{lem:grothendieck}), there is an isomorphism   $\Hom_{\C[\epsilon]}(Q_v, \C)\cong H^0(X,L+\eta)$;    it is not hard to show  that  there is an isomorphism  $Q_v\cong \C[\epsilon]^{m-l}\oplus \C^{l}$ for some $0\le l\le m$. Hence $Q_v$ has a presentation by a $m\times l$ matrix with $\epsilon$ on the diagonal and $0$ elsewhere. A direct computation shows that the $r$-th Fitting ideal is $0$ iff $m>r+1$ or $m=r+1$ and $l=0$. In particular, this proves claim (i) and we may assume from now on that $m=r+1$.

Denote by $L_v$ the pull back of $\cP\otimes \pr_1^*L$ to $X_{\epsilon}:=X\times_{\Spec \C} \Spec \C[\epsilon]$. The condition $l=0$  is equivalent to  the surjectivity of the map $\Hom_{\C[\epsilon]}(Q_v, \C[\epsilon])\to \Hom_{\C[\epsilon]}(Q_v,\C)$. By Lemma \ref{lem:grothendieck},   we have canonical isomorphisms:  $$ \Hom_{\C[\epsilon]}(Q_v, \C[\epsilon])\cong H^0(X_{\epsilon}, L_v), \quad \Hom_{\C[\epsilon]}(Q_v, \C)\cong H^0(X, L+\eta)$$ So $v$ is tangent to $W^r(L,X)$ at $\eta$ iff the restriction map $H^0(X_{\epsilon}, L_v)\to H^0(X,L+\eta)$ is surjective. On the other hand, this map is part of the long cohomology sequence associated with the extension
$$0\to L+\eta \overset{\epsilon}{\to}L_v\to L+\eta\to 0,$$
hence it is surjective iff the coboundary map $H^0( X, L+\eta)\to H^1(L+\eta)$ vanishes. Since it is well known that the latter map is given by  cupping with $v$, statement (ii) follows.
\end{proof}

 As an application  of Proposition \ref{prop:BNtangent} we prove the following:
\begin{prop}\label{prop:dv}
 Let $X$ be a smooth projective variety such that $n:=\dim X=\albdim X$ and let $R$ be the  ramification divisor  of the Albanese map of $X$; if $0<Z\le R$ is a divisor and $s\in H^0(\OO_X(Z))$ is  a section that defines $Z$, then:
\begin{enumerate}
\item the map $H^1(\OO_X)\overset{\cup s}{\longrightarrow}H^1(\OO_X(Z))$ is injective; 
\item if $h^0(Z)=1$, then $H^0(Z|_Z)=0$ and  $0\in W(Z,X)$ is an isolated point (with  reduced structure). 
\end{enumerate}
\end{prop}
\begin{proof} (i) Denote by $\Lambda$ the image of the map $\bigwedge^nH^0(\Omega^1_X)\to H^0(K_X)$; the divisor $R$ is the fixed part of the linear subsystem  $|\Lambda|\subseteq |K_X|$.

Assume for contradiction that  $v\in H^1(\OO_X)$ is a nonzero vector such that $s\cup v=0$; then, since $Z\le R$,  we have $t\cup v=0$ for every $t\in \Lambda$. 

By Hodge theory there exists a nonzero $\beta\in H^0(\Omega^1_X)$ such that $v=\bar\beta$ and the condition $t\cup v=0$ is equivalent to  $t\wedge \bar{\beta}$ being  an exact form: $t\wedge \bar{\beta}=d\phi$.
 Let now $x\in X$ be a point such that $\beta(x)\ne 0$ and such that the differential of the Albanese map at $x$ is injective. Then we can find $\alpha_1,\dots \alpha_{n-1}\in H^0(\Omega^1_X)$ such that $\alpha_1,\dots \alpha_{n-1}, \beta$ span the cotangent space $T^*_xX$. Hence the form   $t:=\alpha_1\wedge\dots \wedge \alpha_{n-1}\wedge \beta$ is nonzero at $x$ and therefore  $t\wedge \bar t\ne 0$ and 
 $(-i)^n\int_X t\wedge \bar t > 0.$ On the other hand, we have 
 $$\int_X t\wedge \bar  t =\pm\int_X( t\wedge \bar \beta)\wedge \bar\alpha_1\wedge\dots \wedge \bar\alpha_{n-1}=\int _X d(\phi \wedge \bar\alpha_1\wedge\dots \wedge \bar\alpha_{n-1})=0 ,$$ a contradiction.   So $H^1(\OO_X)\overset{\cup s}{\longrightarrow} H^1(\OO_X(R))$ is injective.
\smallskip

(ii)   By  (i) and by Proposition \ref{prop:BNtangent}, the tangent space to $W(Z,X)$ at $0$ is zero, hence $\{0\}$ with reduced structure  is a component of $W(Z,X)$. 
   The vanishing of $H^0(Z|_Z)$ also follows   by (i)  taking cohomology in the usual restriction sequence $0\to \OO_X\to\OO_X(Z)\to \OO_Z(Z)\to 0$. \end{proof}

\begin{cor}\label{cor:fixed-rigid}
Let $X$ be a smooth projective variety such that $n:=\dim X=\albdim X$ and let $Z>0$ be a divisor contained in the fixed part of $|K_X|$. 
Then  $H^0(Z|_Z)=0$ and  $0\in W(Z,X)$ is an isolated point.
\end{cor}
\begin{proof}   As usual, let $R$ denote the ramification divisor of the Albanese map of $X$. Since   $Z$ is contained in the fixed part of  $|K_X|$,  we have  $Z\subseteq R$ and $h^0(Z)=1$.  So the claim follows by  Proposition \ref{prop:dv}.\end{proof}

\section{Restriction maps }\label{sec:restriction}

In this section we consider a smooth projective variety $X$ of maximal Albanese dimension and  an effective divisor $D\subset X$ and we establish  lower bounds for the rank of the restriction  map  $$r_D\colon H^0(K_X)\to H^0(K_X|_D),$$  and for the corank of the residue map 
$$\res_D\colon H^0(K_X+D)\to H^0(K_D).$$

Such bounds, besides being intrinsically interesting, can be used  to give lower bounds for the arithmetic genus of divisors moving in a positive dimensional linear system (Corollaries \ref{cor:genere} and \ref{cor:genere2}).

More precisely,  we give three inequalities. The first two (Proposition \ref{prop:xiao}) are uniform bounds for the rank of $r_D$ and the corank of $\res_D$ under the assumptions that $D$ is irreducible, not contained in the ramification locus of the Albanese map $a\colon X\to \Alb(X)$,  and $a(D)$ generates $\Alb(X)$. This is one of the   ingredients in  the proof of Theorem \ref{thm:Cmoves}, which is our main result on the structure of Brill-Noether loci in the case of surfaces.  

The third one (Proposition \ref{prop:xiao-petri})  is based on the infinitesimal analysis of the Brill-Noether locus $W^r(D,X)$  carried out in \S \ref{sec:BNdef}: the bound that we obtain is stronger than that of Proposition \ref{prop:xiao} but it   requires further assumptions.

\subsection{Preliminary results}\label{ssec:preliminari}
The main goal of this section is to prove Proposition \ref{prop:rank}, which is the key result that enables us to obtain the inequalities of \S \ref{ssec:xiao}. 

We start by listing some well known facts of linear   algebra:
\begin{lem}[Hopf lemma]\label{lem:hopf}
Let $U,V$ and $W$ be complex vector spaces of finite dimension and let $f\colon U\otimes V\to W$ be a linear map. 

If $\ker f$ does not contain any nonzero simple tensor $u\otimes v$, then $\rk f\ge \dim U+\dim V-1$. 
\end{lem}

\begin{lem}\label{lem:grass}
Let $V$, $W$ be complex vector spaces of finite dimension and let $f\colon \bigwedge^kV\to W$ be a linear map. 

If $\ker f$ does not contain any nonzero simple tensor $v_1\wedge \dots \wedge v_k$, then $\rk f\ge k(\dim V-k)+1$.
\end{lem}

\begin{lem}[$\ker$/$\coker$ lemma] \label{lem:ker-coker}
Let $V$, $W$ be complex vector spaces of finite dimension and let $f, g\colon V\to W$ be  linear maps.

If $\rk(f+tg)\le \rk f$ for  every $t\in \C$, then $g(\ker f)\subseteq \im f$. 
\end{lem}

The next  result is possibly  also known, but  since it is less obvious we give a proof for completeness.

\begin{lem} \label{lem:xiaode}
Let $V$, $W$ be complex vectors spaces of finite dimension and set $q:=\dim V$.  Let 
  $\phi: \bigwedge^2 V \to W$ be a linear map such that: 
    \begin{enumerate}[(a)] 
  \item for every  $0\ne v\in V,$ there exists  $w\in V$ such that $\phi(v\wedge w)\neq 0$;
  \item if  $\phi(v\wedge w)=\phi(v\wedge u)=0$ and $v \neq 0$,   then $\phi(u\wedge w)=0$.
    \end{enumerate}
 Then:
      \begin{enumerate}
\item $\dim \phi(V) \geq q-1.$
\item  there exists $v\in V$ such that the restriction of $\phi$ to $v\wedge V$ is injective.     
 \end{enumerate}
  \end{lem}
  \begin{proof}
  We observe first of all that (i) follows from (ii), hence it is enough to prove (ii).
  
  For every $v\in V$ we let $k_v\colon V\to W$ be the linear map  defined by $x\mapsto \phi(v\wedge x)$,  and we let $U(v)$ be the kernel  and $S(v)$ be the image  of $k_v$.
 Of course $v\in U(v)$, and for  $v\ne 0$ the assumptions give: 
    \begin{enumerate}[(1)]
  \item  $U(v)\subsetneq V$; 

 \item  $U(v)=U(v') \iff  v'\in U(v)\setminus \{0\}$.   
 \end{enumerate}
 
  Claim (ii) is equivalent to the existence of a vector $v\in V$ such that $U(v)$ is $1$-dimensional. Choose $v$ with $m:=\dim U(v)$ minimal; notice that we have $0<m<q$. For any  vector $u\in V$ and any $t\in \C$ the map $k_v+tk_u=k_{v+tu}$ has rank $\le q-m$; by Lemma \ref{lem:ker-coker} we have $k_u(U(v))\subseteq S(v)$. Hence if $\phi(v'\wedge v)=0$, then for any $u\in V$ there exists $h\in V$ such that 
  $$k_{u}(v') =\phi(u\wedge v')=\phi(v\wedge h)=k_v(h).$$
Since $k_{u}(v')=-k_{v'}(u)$, it follows that   for every $v'\in U(v)$ we have $S(v')\subseteq S(v)$; since $U(v)=U(v')$ by (2), it follows $S(v)=S(v')$. Let now $L\subset V$ be a subspace such that $V=U(v)\oplus L$. Then for every $0\ne v'\in U(v)$, $k_{v'}$ restricts to an isomorphism $h_{v'}\colon L\to S(v)$.   Fixing bases for $L$ and $S(v)$, this isomorphism is represented by an invertible matrix of order   $q-m> 0$, whose entries depend linearly on $v'$.  Then taking determinants one obtains a homogeneous polynomial  of degree $q-m>0$ that has no zeros  on $\pp(U(v))$. Since  we are working over an algebraically closed field, this is possible only if $\dim U(v)=1$. 
  \end{proof}
Given a vector bundle $E$ on a variety $X$ and a finite dimensional subspace $V\subseteq H^0(X,E)$, for any integer $k\ge 0$  we denote by $\psi_k\colon \bigwedge^kV \to H^0(X,\bigwedge^kE)$ the natural map.  Here is the main result of this section:
  \begin{prop}\label{prop:rank} Let $X$ be an irreducible  variety, let $E$ be a rank $n$ vector bundle on $X$. Assume that there exists a subspace  $V\subseteq H^0(X,E)$ of dimension $q$ that generates $E$ generically. 
  
  Then the  map  $\psi_n\colon  \bigwedge^n V \to H^0(X,\det E)$ has rank $\ge q-n+1$.
  \end{prop}
  \begin{proof} The proof is by induction on the rank $n$ of $E$, the case $n=1$ being trivial.

Up to restricting to a Zariski open set, we may assume that $X$ is affine and that $V$ generates $E$. 

Consider first  the map 
$\psi_2\colon  \bigwedge^2 V\to W:=H^0(X,\bigwedge^2E)$. Since  $\psi_2$ satisfies the assumptions of Lemma \ref{lem:xiaode}, there exist a section $s\in V$ such that $\psi_2(s\wedge t)=0$ if and only if $t=\lambda s$ for some $\lambda \in \C$. 
Up to replacing $X$ by an open subset, we may assume that $s$ vanishes nowhere on  $X$, hence there is a short exact sequence 
\begin{equation}\label{eq:EE'}
0\to\OO_X\overset{s}{\to}E\to E'\to 0,\end{equation}
with $E'$ a rank $n-1$ vector bundle.  We denote by $V'\subseteq H^0(X, E')$ the image of $V$; the subspace $V'$ has dimension $q-1$ and generates $E'$ on $X$, hence by the inductive assumption the map $\psi'_{n-1}\colon \bigwedge^{n-1}V'\to H^0(X,\det E' )$ has rank $\ge (q-1)-(n-1)+1=q-n+1$. 

The sequence \eqref{eq:EE'} induces an isomorphism $\det E\to \det E'$ and the induced map $H^0(X,\det E)\to H^0(X,\det E')$  maps $\im\psi_n$ to a subspace containing $\im\psi'_{n-1}$. Hence $\rk \psi_n\ge \rk \psi'_{n-1}\ge q-n+1$.
\end{proof}

\subsection{Uniform bounds }\label{ssec:xiao}

Here we use the results of \S \ref{ssec:preliminari} to bound the rank  of the map $r_D\colon H^0(K_X)\to H^0(K_X|_D)$ and the corank of the residue map $\res_D\colon H^0(K_X+D)\to H^0(K_D)$, where $D$ is   an effective divisor of an irregular variety $X$.

\begin{prop}\label{prop:xiao}  Let $X$ be a smooth projective variety with $\albdim X=\dim X=n$ and let $D>0$ be an irreducible divisor of $X$ such  that the image  of $D$ via the Albanese map $a\colon X\to  \Alb(X)$ generates $\Alb(X)$. Assume that $D$ is not contained in the ramification divisor of $a$. 
Then, letting $q:=q(X)$:
\begin{enumerate}
\item the rank of $r_D\colon H^0(K_X)\to H^0(K_X|_D)$ is $\ge q-n+1$;
\item the corank of $\res_D\colon H^0(K_X+D)\to H^0(K_D)$  is $\ge q-n+2$.
\end{enumerate}
\end{prop}
\begin{proof}
(i) The inequality follows by Proposition \ref{prop:rank} taking $E=\Omega^1_X|_D$ and  $V=i^*H^0(X,  \Omega^1_X)$, where $i\colon D\to X$ is the inclusion.
\medskip

(ii) Consider the short exact sequence $0\to K_X\to K_X+D\to K_D\to 0$. Taking cohomology, we see that the corank of $\res_D$ is equal to the rank of the coboundary map $\partial \colon H^0(K_D)\to H^1(K_X)$ or, taking duals, to the rank of ${}^t\partial \colon H^{n-1}(\OO_X)\to H^{n-1}(\OO_D)$.

Now let $(X',D')$ be an embedded resolution  of $(X,D)$; then there is  a commutative diagram:
\[
\begin{CD}
H^{n-1}(\OO_{X})@>>>H^{n-1}(\OO_{X'})\\
@V{{}^t\partial}VV  @V{{}^t\partial}VV \\
H^{n-1}(\OO_{D})@>>>H^{n-1}(\OO_{D'})
\end{CD}
\]
where the top horizontal map is an isomorphism. Hence we may assume
without loss of generality  that $D$ is smooth. 

Then, by Hodge theory, the map ${}^t\partial$ is the complex conjugate of the natural  map $\rho\colon H^0(\Omega^{n-1}_X)\to H^0(K_D)$. Here we set $E=\Omega^1_D$ and  $V=i^*H^0( \Omega^1_X)\subseteq H^0(\Omega^1_D)$, where $i\colon D\to X$ is the inclusion; then  the image of $\rho$ contains the image of $\psi_{n-1}\colon \bigwedge^{n-1}V\to H^0(K_D)$. The required inequality now follows by Proposition \ref{prop:rank}, since $V$ has dimension $q$ by the assumption that $a(D)$ generates $\Alb(X)$.
\end{proof}

Statement (i) of Proposition \ref{prop:xiao} has been proven in the case of  surfaces fibered over $\pp^1$  by Xiao Gang in \cite{xiaoirreg}. The following corollary generalizes to arbitrary dimension the  main result of  \cite{xiaoirreg}. 

\begin{cor}\label{cor:genere}
  Let $X$ be a smooth projective variety with $\albdim X=\dim X=n$ and let $D>0$ be an irreducible divisor of $X$.  If $h^0(D)=r+1\ge 2$, 
  then: $$h^0(K_D)\ge 2(q+1-n)+r.$$
\end{cor}
\begin{proof}
By the semicontinuity of $h^0(K_{D'})$ as $D'\in |D|$ varies, we may replace  $D$ by a general element of $|D|$, and  assume that $D$ is not contained in the ramification locus of the Albanese map $a\colon X\to \Alb(X)$. Observe also that $a(D)$ generates $\Alb(X)$, since $D$ moves linearly.

Consider the map  $\res_D\colon H^0(K_X+D){\lra}H^0(K_D)$. 
By Proposition \ref{prop:xiao}, we have $h^0(K_D)\ge \rk \res_D+q-n+2$.

To give a bound on the rank of $ \res_D$, we observe that the image of $\res_D$ contains the image of the multiplication map 
$$(\im r_D)\otimes H^0(D)/\!<s>\lra H^0(K_D),$$
where $s\in H^0(D)$ is a section defining $D$.  By Proposition \ref{prop:xiao},  $\rk  r_D\geq q-n+1$ and so applying  Lemma \ref{lem:hopf}  we obtain $\rk \res_D\ge q-n+r$. Hence $h^0(K_D)\ge (q-n+r)+q-n+2=2(q+1-n)+r$.
\end{proof}

For future reference, we observe the following:
  \begin{cor}\label{cor:CC} Let $S$ be a smooth complex surface with $\albdim(S)=2$,   let $a\colon S\to \Alb(S)$ be the Albanese map and let $C\subset S$ be a 1-connected  curve  having a component $C_1$ not contained in  the ramification locus of  $a$ and such that  $C_1^2>0$. Then:
 \begin{enumerate}
\item $h^0(K_S-C)\leq \chi(S)$;
 
\item $h^0(C|_C)\ge q+C^2-p_a(C).$
\end{enumerate}

  \end{cor}
  \begin{proof} (i)  Note that  $C_1$ is nef and big; therefore  $h^1(\OO_S(-C_1))=0$, the map  $H^1(\OO_S)\to H^1(\OO_{C_1})$ is an injection and $a(C_1)$ generates $\Alb(S)$. Since, by Proposition  \ref{prop:xiao} we have $ \rk r_{C_1}\ge q-1$, we obtain $\rk r_{C}\ge q-1$ and so $h^0(K_S-C)\leq \chi(S)$.
  \medskip
  
(ii) Since $C$ is $1$-connected, Riemann--Roch on $C$ gives:
  $$h^0(C|_C)= C^2+h^0(K_S|_C) +1-p_a(C)\ge C^2 + q-p_a(C).$$
   \end{proof}

 \subsection{The Petri map}\label{ssec:petri}

 Let $X$ be a smooth projective variety and let $D$ be an effective divisor on $X$.
 
As a  tool for   studying  the rank or $r_D$  we introduce  the  {\em Petri map}, which,  in analogy with the case of curves, is the map 
$$\beta_{D}\colon H^0(K_X-D)\otimes H^0(D)\to H^0(K_X),$$
induced by cup product. 

The Petri map is strictly related  to the infinitesimal structure of the Brill-Noether loci, as follows.    Let $r=h^0(D)-1$,  let $T$ be the tangent space to $W^r(D,X)$  at $0$ and   let $\alpha\colon H^1(\OO_X)\otimes H^{n-1}(\OO_X)\to H^{n}(\OO_X)$ be the map   induced by cup product. 
Then by Proposition \ref{prop:BNtangent} for all $\sigma\in T\otimes H^{n-1}(\OO_X)$ and $\psi \in H^0(D)\otimes H^0(K_X-D)$ we have
$$\alpha(\sigma)\cup\beta_{D}(\psi)=0,$$ 
namely $V:=\alpha(T\otimes H^{n-1}(\OO_X))\subseteq H^{n}(\OO_X)$ is orthogonal to $\im \beta_D\subseteq H^0(K_X)$.

\begin{prop} \label{prop:xiao-petri}
Let $X$ be a smooth projective variety of dimension $n$ and irregularity $q\ge n$, let $D> 0$ be a divisor of $X$,  let  $r:=h^0(D)-1$ and  let $T$ be the  tangent space to $W^r(D,X)$ at the point $0$. Assume that  $\dim T>0$ and that $X$ has no fibration $f\colon X\to Z$, with $Z$ normal of  Albanese general type and   $0<\dim Z<n$; 
then:
\begin{enumerate}
\item if $h^0(K_X-D)=0$, then $\rk r_D\geq n(q-n)+1$;
\item if $h^0(K_X-D)>0$, then: 
\begin{align*}\rk r_D \geq (n-1)(q-n)+\dim T+r,\ \mbox{if}\  \dim T\le q-n,\\
\rk r_D\geq n(q-n)+1+r, \qquad \quad \mbox{if}\   \dim T\ge q+1-n.
\end{align*}
\end{enumerate}
\end{prop} 
\begin{proof} 
(i) If $h^0(K_X-D)=0$ then $\rk r_D=h^0(K_X)$. By  Theorem \ref{thm:cast-cat}, under our assumptions the map $\bigwedge^nH^0(\Omega^1_X)\to H^0(K_X)$ does not map any simple tensor to $0$, hence Lemma \ref{lem:grass} gives  $h^0(K_X)\ge n(q-n)+1$.
\smallskip

(ii) Let $\alpha\colon H^1(\OO_X)\otimes H^{n-1}(\OO_X)\to H^{n}(\OO_X)$  be the map induced by cup product. As we have remarked at the beginning of the section, 
$V:=\alpha(T\otimes H^{n-1}(\OO_X))\subseteq H^{n}(\OO_X)$ is orthogonal to $\im \beta_D\subseteq H^0(K_X)$, where $\beta_D$ is the Petri map. Let ${\mathbb G}_T\subseteq {\mathbb G}(n,H^1(\OO_X))$ be the subset consisting of the subspaces that have non trivial  intersection with $T$. Since, as we remarked in (i),  the map $\bigwedge^nH^0(\Omega^1_X)\to H^0(K_X)$ does not map any simple tensor to $0$,  the complex conjugate map $\bigwedge^nH^1(\OO_X)\to H^n(\OO_X)$ induces a morphism ${\mathbb G}_T\to \pp(V)$ which is finite onto its image. It follows that  $\dim V\ge \dim {\mathbb G}_T+1$.
Since, as noticed above,  the space  $V$ is orthogonal to $\im \beta_D$, the   codimension of $\im \beta_D$ is  $\ge \dim  {\mathbb G}_T+1$.

 On the other hand, by Lemma \ref{lem:hopf} the dimension of   $\im \beta_D$  is at least  $h^0(K_X-D)+h^0(D)-1$. Since  $h^0(K_X-D)=p_g-\rk r_D$   and $h^0(D)=r+1$, one obtains  that the codimension of $\im \beta_D$ is $\le \rk r_D-r$.
 So $\rk r_D\ge \dim {\mathbb G}_T+r+1$, which is precisely the statement.
\end{proof}
Arguing as in the proof of Corollary \ref{cor:genere}, one obtains   the following:
\begin{cor}\label{cor:genere2} 
Let $X$ be a smooth projective variety of dimension $n$ and irregularity $q\ge n$ that has no fibration $f\colon X\to Z$, with $Z$ normal of  Albanese general type and   $0<\dim Z<n$. Let $D> 0$ be a divisor of $X$,  let $r:=h^0(D)-1$ and  let $T$ be the  tangent space to $W^r(D,X)$ at the point $0$. Assume that  $r>0$ and $\dim T>0$; then
\begin{enumerate}
\item if $h^0(K_X-D)=0$, then $h^0(K_D)\geq (n+1)(q-n)+r+2$;
\item if $h^0(K_X-D)>0$, then: 
\begin{align*}h^0(K_D)\geq n(q-n)+\dim T+2r+1,\qquad \mbox{if}\  \dim T\le q-n,\\
h^0(K_D)\geq (n+1)(q-n)+2r+2, \qquad \quad \mbox{if}\   \dim T\ge q+1-n.
\end{align*}
\end{enumerate}

  \end{cor}

 \section{Brill-Noether theory for singular curves}\label{sec:BN}
 
Here we prove a   generalization to the case of a compact subgroup of the Jacobian of  a reduced connected   curve of the classical results on the Brill-Noether loci of smooth curves. The results of this section are used in section \ref{sec:Cmoves} to prove the analogous result for smooth  irregular surfaces (Thm. \ref{thm:Cmoves}).
  \smallskip
  
  Assume that $C$ is a reduced connected  projective curve with irreducible components $C_1, \dots C_k$; for every $i$, denote by $\nu_i\colon C_i\unu\to C_i$ the normalization map. 
  We refer the reader to  \cite[\S 9.2, 9.3]{neron} for a detailed  description of  the Jacobian $\Pic^0(C)$. 
We just recall here   that $\Pic^0(C)$ is a smooth algebraic group  and that there is an  exact sequence:
$$  0\to G\to \Pic^0(C)\overset{f}{\to} \Pic^0(C_1\unu)\times \dots\times  \Pic^0(C_k\unu)\to 0,$$
  where $G$ is a smooth connected linear algebraic group and $f(\eta)=(\nu_1^*\eta,\dots, \nu_k^*\eta)$.     
   Notice that if $T\subseteq\Pic^0(C)$ is a complete subgroup, then $G\cap T$ is a finite group, and therefore  the induced map $T\to  \Pic^0(C_1\unu)\times \dots \times \Pic^0(C_k\unu)$ has finite kernel. 
  
  Fix $L \in \Pic(C)$, an integer $r\ge 0$  and a complete connected subgroup  $T\subseteq \Pic^0(C)$,  and consider  the  Brill-Noether locus $W_T^r(L,C):=\{\eta\in T\ |\ h^0(L\otimes \eta)\ge r+1\}$.
   
As in the case of a smooth curve $C$, we  define the Brill-Noether number $\rho(t, r, d):=t-(r+1)(p_a(C)-d+r)$, where $d$ is the total degree of $L$ and $t=\dim T$. In  complete analogy with the classical situation, we prove:

\begin{thm}\label{thm:BN} Let $r\ge 0$ be an integer, let  $C$ be a reduced connected projective curve and let $L$ be a line bundle on $C$ of total degree $d$.   If \/ $T\subseteq\Pic^0(C)$ is a complete connected subgroup of dimension $t$  such that  for every component $C_i$ of $C$ the map $T\to \Pic^0(C_i\unu)$ has finite kernel, then: 
\begin{enumerate}
\item if $\rho(t, r, d)\ge 0$, then  $W_T^r(L,C)$ is nonempty;
\item if  $\rho(t, r, d)> 0$, then $W_T^r(L,C)$ is connected, it  generates $T$   and  each of its  components has dimension $\ge\min\{ \rho(t, r, d), t\}$.
\end{enumerate}
 \end{thm}
 \begin{proof}
 The proof  follows closely the  proof given  by Fulton and Lazarsfeld in the case of a smooth curve (cf. \cite{FL}, \cite[\S 6.3.B, 7.2]{lazarsfeld}).  
 
 Denote by $\cP$ the restriction to $C\times T$ of a normalized Poincar\'e line bundle  on $C\times \Pic^0(C)$. Let   $H$ be a sufficiently high multiple of an ample line bundle of $C$ and let $M:=L\otimes H$.  Recall that for any product of varieties we denote by $\pr_i$ the projection onto the $i$-th factor;  we define:
$$E:={\pr_2}_*(\pr_1^*M\otimes \cP).$$
By the choice of $H$, $E$ is a vector bundle of rank $d+\deg H+1-p_a(C)$ on $T$ and for every $\eta\in T$ the natural map $E\otimes\C(\eta)\to H^0(M\otimes\eta)$ is an isomorphism and $M\otimes \eta$ is generated by global sections. 

We let $Z=x_1+\dots + x_m\in |H|$ be  a general divisor and we  set $F:={\pr_2}_*(\pr_1^*M|_Z\otimes\cP)$.  The sheaf $F$ is a vector bundle of rank $m=\deg H$ on $\Pic^0(C)$ and 
  the evaluation map $\pr_1^*M\otimes \cP\to \pr_1^*M|_Z\otimes\cP$  induces a  sheaf map $E\to F$. The locus  where this  map drops rank by $r+1$ is  $W^r_T(L,C)$.   
  
   By Theorem II and Remark 1.9 of \cite{FL}, to prove the theorem it suffices  to show  that $\Hom(E,F)$ is an ample vector bundle. 
  We have $F=\oplus_i\cP_{x_i}$, where $\cP_{x_i}$ is (isomorphic to) the restriction of $\cP$ to $\{x_i\}\times T$. Since $\cP$ is the restriction of a normalized  Poincar\'e line bundle, $\cP_{x_i}$ is algebraically equivalent to $\OO_T$.  Hence $\Hom(E,F)=\oplus_{i=1}^n(E^{\vee}\otimes\cP_{x_i})$ is ample if and only if $E^{\vee}$ is ample. 
  
To show the ampleness of $E^{\vee}$ we adapt   the proof of \cite[Thm. 6.3.48]{lazarsfeld}. 
Denote by $\xi$ the linear  equivalence class of the tautological line bundle of  $\pp(E^{\vee})$; we are going to show that for any irreducible positive dimensional subvariety $V$ of  $\pp(E^{\vee})$   the cycle $V\cap \xi$ is represented, up to numerical equivalence, by a proper nonempty subvariety of $V$.

Given a point $x\in C$, the evaluation map $E\to \cP_x$ is surjective, since $M\otimes\eta$ is globally generated for every $\eta\in T$, hence it defines an effective divisor $I_x$ algebraically equivalent to $\xi$.
Denote by $p\colon \pp(E^{\vee})\to T$ the natural projection. A point   $v\in \pp(E^{\vee})$ is determined by  a section $s_v\in H^0(M\otimes p(v))$, and $v\in I_x$ if and only if $s_v(x)=0$. Let $C_i$ be  a component of $C$ such that the general element of $V$ does not vanish identically on $C_i$. If the support of the zero locus of  $s_v$ on $C_i$ varies, then for a general $x\in C_i$ the set $V\cap I_x$ is a proper nonempty subvariety algebraically equivalent to  $V\cap \xi$ and we are done. So assume  that for general $v\in V$   the support of the zero locus of $v$ on $C_i$ is constant: then,  pulling back via $\nu_i\colon C_i\unu\to C_i$, we see that the line bundle  $\nu_i^*(M\otimes p(v))$ stays  constant as $v\in V$ varies. Since the map $T\to \Pic^0(C_i\unu)$ has finite kernel by assumption, $p(V)$ is a point $\eta_V\in T$ and $V\subseteq \pp(H^0(M\otimes\eta_V))$. Since $\dim V>0$ and $\xi$ restricts to the class of a hyperplane of $ \pp(H^0(M\otimes\eta_V))$, the cycle $V\cap \xi$ is represented by a proper nonempty subvariety of $V$ also in this case. This completes the proof.
 \end{proof}
 
\begin{rem} \label{rem:ipotesi} 
The proof of Theorem  \ref{thm:BN} does not extend to the case of a  complete subgroup $T\subseteq \Pic^0(C)$  such  that the the map $T\to \Pic^0(C_i)$ does not have finite kernel for some component $C_i$ of $C$. Indeed, take $C=C_1\cup C_2$, with $C_i$ smooth curves of genus $g_i>0$ meeting transversely at only one point $P$, and $T=\Pic^0(C_1)\subset \Pic^0(C)=\Pic^0(C_1)\times \Pic^0(C_2)$.  Twisting by  $H\otimes \eta$, $\eta\in T$, the exact sequence
$0\to \OO_{C_2}(-P)\to \OO_C\to \OO_{C_1}\to 0$
and taking global sections,  one gets inclusions $$H^0( \OO_{C_2}(H-P))=H^0( \OO_{C_2}(H-P)\otimes\eta)\hookrightarrow H^0( \OO_C(H)\otimes \eta)$$ that sheafify to a vector bundle map 
 $\OO_T\otimes H^0(\OO_{C_2}(H-P))\to E$. So the bundle $E^{\vee}$ is not ample.
 
We do not know whether the statement of Theorem \ref{thm:BN} still holds without this assumption on $T$.
\end{rem}

\section{Brill-Noether theory for curves on irregular surfaces}\label{sec:Cmoves}

Our approach to the study of the Brill-Noether loci $W^r(D, X)$ for an effective divisor $D$ in an $n$-dimensional variety $X$  of maximal Albanese dimension consists in comparing it with  a suitable Brill-Noether locus on the $(n-1)$-dimensional variety $D$. 
Let $i^*\colon \Pic^0(X)\to \Pic^0(D)$ be the map induced by the inclusion $i\colon D\to X$ and denote by $T$ the image of $i^*$. The key observation is the following:

\begin{prop}\label{prop:restriction} Let $X$ be a variety of dimension $n>1$ with $\albdim X=n$ and  without irrational pencils of genus $>1$ and let   $D> 0$ be a divisor of $X$.  Let $Y$ be a positive dimensional irreducible  component of 
$W^r_T(D|_D, D)$; if $\dim Y\ge 2$ or $0\in Y$, then ${i^*}\inv Y$ is a component of $W^r(D,X)$.
\end{prop}
\begin{proof} Let $V^1(X)\subset \Pic^0(X)$ the first  Green-Lazarsfeld locus, namely  $V^1(X)=\{\eta\in \Pic^0(X)|h^1(\eta)>0\}$ (see \S \ref{ssec:GV}).

Denote by $U$ the complement of $V^1(X)$ in $\Pic^0(X)$; for $\eta\in U$, the short exact sequence:
$$0\to \eta \to \OO_X(D+\eta)\to (D+\eta)|_D \to 0$$
induces an isomorphism $H^0(\OO_X(D+\eta))\cong H^0((D+\eta)|_D)$. Hence $U\cap W^r(D, X)=U\cap{ i^*}\inv W^r_T(D|_D, D)$ and  to prove the claim it is enough to show that ${i^*}\inv Y\not\subset V^1(X)$. By Theorem \ref{thm:GV1}, if $\dim Y\ge 2$ this follows by the fact that $\dim V^1(X)\le 1$ and  if $0\in Y$ this  follows by the fact that $0$ is an isolated point of $V^1(X)$.
\end{proof}

In the case of surfaces, Proposition \ref{prop:restriction} can be made effective.  

Let $S$ be a surface with $q(S)=q$ and let $C\subset S$ be a curve; we define the Brill-Noether number  $\rho(C,r):=q-(r+1)(p_a(C)-C^2+r)$. For $r=0$ we write simply  $\rho(C)$ for $\rho(C,0)=q+C^2-p_a(C)$.  Recall that by the adjunction formula  $q+C^2-p_a(C)=q-1+\frac{C^2-K_SC} {2}$.

 \begin{thm}\label{thm:Cmoves} Let $r\ge 0$ be an integer. Let $S$ be a surface with irregularity $q>1$ that has no irrational pencil of genus $>1$ and let $C\subset S$ be a reduced   curve such that  $C_i^2>0$ for every  irreducible component  $C_i$ of $C$. 
 \begin{enumerate}
\item If    $\rho(C,r)>1$ or   $\rho(C,r)=1$ and $V^1(S)=\{\eta\in \Pic^0(S)|h^1(\eta)>0\}$ does not generate $\Pic^0(S)$, then $W^r(C,S)$ is nonempty  of dimension $\ge\min\{q,  \rho(C, r)\}$. 

 \item If  $\rho(C)>1$, or  $\rho(C)=1$ and $C$ is not contained in the ramification locus of the Albanese map, 
or $\rho(C)=1$ and $V^1(S)$ does not generate $\Pic^0(S)$,
 then $W(C,S)$ has an irreducible component of dimension  $\ge \min\{q,  \rho(C\})$ containing $0$.
  \end{enumerate}  
  \end{thm}

  \begin{proof} We start by observing that  by the Hodge index theorem any two irreducible components of $C$ intersect, hence in particular $C$ is connected. 
  
    Let $C_i$ be a component of $C$ and denote by $C_i\unu$  its normalization; since $C_i^2>0$, by   \cite[Prop. 1.6]{cfm} the map $\Pic^0(S)\to \Pic^0(C_i)$ is an injection. Since $\Pic^0(S)$ is projective and the kernel of $\Pic^0(C_i)\to \Pic^0(C_i\unu)$ is an affine algebraic group, it follows that the map $\Pic^0(S)\to \Pic^0(C_i\unu)$ has finite kernel and we may apply Theorem \ref{thm:BN}.

 \smallskip 
 
  By Theorem \ref{thm:BN}, if $\rho(C,r)>0$ then  $W^r_{\Pic^0(S)}(C|_C,C)$ is nonempty, it generates $\Pic^0(S)$ and all its components have dimension $\ge \min\{q,  \rho(C, r)\}$. Claim (i) follows directly by Proposition \ref{prop:restriction} if $\rho(C,r)>1$.  If  $\rho(C,r)=1$ and $V^1(S)$ does not generate $\Pic^0(S)$,  then there exist a positive dimensional component $Y$ of $W^r_{\Pic^0(S)}(C|_C,C)$ not contained in $V^1(S)$  and arguing as   in the proof of Proposition \ref{prop:restriction} one shows that $Y$ is a component of $W^r(C,S)$.
  
 \smallskip 
 
 By Proposition \ref{prop:restriction} to prove claim (ii)  it is enough to show that $0\in W_{\Pic^0(S)}(C|_C,C)$, namely that $h^0(C|_C)>0$.

 If $\rho(C)=1$ and $C$ is not contained in the ramification locus of the Albanese map of $S$,  this follows by   Corollary \ref{cor:CC}.

Otherwise assume that  $\rho(C)> 1$ or $\rho(C)=1$ and $V^1(S)$ does not generate $\Pic^0(S)$. Then   by claim (i), $(-1)^*W(C,S)$ has dimension $\ge \min\{q,  \rho(C)\}$.   As previously we conclude that the hypotheses $\rho(C)> 1$ or $\rho(C)=1$ and $V^1(S)$ does not generate $\Pic^0(S)$ also imply that
$(-1)^*W(C,S)$ is not contained in  $V^1(S)$.

Assume for contradiction that $h^0(C|_C)=0$. Then the Riemann-Roch theorem on $C$ gives $h^0(K_S|_C)=p_a(C)-C^2-1$. Since $p_a(C)-C^2-1=q-1-\rho(C)$, one obtains  $h^0(K_S|_C)<q-1$ and thus  $h^0(K_S-C)>\chi(S)$.  

 For  every $\eta\in W(C,S)$ we have $h^0(K_S+\eta)\ge h^0(K_S+\eta-(C+\eta))=h^0(K_S-C)>\chi(S)$, hence $-\eta\in V^1(S)$,  a contradiction. This completes the proof. 
\end{proof}

 \begin{rem}\label{rem:mild}  There are plenty of  irregular surfaces without irrational pencils,  for instance complete intersections in  abelian varieties and symmetric products of curves (cf. \cite[\S 2]{survey}); indeed such surfaces  can be regarded in some sense as ``the general case''.
 
  Note that if $S$ has no irrational pencil of genus $>1$ and   $\Alb(S)$ is not isogenous to a product of elliptic curves, then the assumption that $V^1(S)$ does not generate $\Pic^0(S)$ is verified, since by Theorem \ref{thm:GV1}  the positive dimensional components of $V^1(S)$ are  elliptic curves. 
In Example \ref{ex:V1} we describe  a  surface without irrational pencils of genus $>1$ such that $V^1(S)$ generates $\Pic^0(S)$. 
  
  Furthermore  the inequalities of Theorem \ref{thm:Cmoves} are sharp: see Example \ref{ex:sym}. 
  
   \end{rem}

\section{Applications to  curves on surfaces of maximal Albanese dimension}\label{sec:applications}

\subsection{Curves that do not move  in a linear series}
Here we apply the results of the previous sections  to  curves $C$ with $h^0(C)=1$  on a  surface of general type $S$. 

The cohomology sequence associated to the restriction sequence for such a curve $C$ gives an exact sequence: 
$$0\to H^0(C|_C)\lra H^1(\OO_S)\overset{\cup s}{\lra} H^1(\OO_S(C)),$$
where  $s\in H^0(\OO_S(C))$ is a nonzero section vanishing on $C$. Hence by Proposition \ref{prop:BNtangent}, the space $H^0(C|_C)$ is naturally isomorphic to the tangent space to $W(C,S)$. This remark, together with Proposition \ref{prop:xiao-petri}, gives the    following:
\begin{lem}\label{prop:ineq1}
 Let $S$ be a surface of general type with irregularity $q>0$  that has no irrational pencil of genus $>1$  and let $C\subset S$ be a $1$-connected curve with $h^0(C)=1$. Then one of the following occurs:
\begin{enumerate}
\item $0\in W(C,S)$ is an isolated point (with reduced structure);
\item$0<h^0(C|_C)<q$ and  $C^2+2q-4\le K_SC$;
\item $h^0(C|_C)=q$ and  $C^2+2q-6\le K_SC$.
\end{enumerate}

\end{lem}
\begin{proof} As we observed above, the tangent space to $W(C,S)$ has dimension equal to $h^0(C|_C)$, therefore case (i) occurs for $h^0(C|_C)=0$. If $h^0(C|_C)>0$, then we can apply Proposition \ref{prop:xiao-petri}, which gives 
$h^0(K_S|_C)\ge q-2+h^0(C|_C)$ if $h^0(C|_C)<q$ and $h^0(K_S|_C)\ge 2q-3$ if $h^0(C|_C)=q$.  By Riemann-Roch and by the adjunction formula, we have:
 $$h^0(K_S|_C)=h^0(C|_C)+K_SC+1-p_a(C)= h^0(C|_C)+\frac{K_SC-C^2}{2},$$
and statements (ii) and (iii) follow immediately by plugging  this expression in the above inequalities.
\end{proof}
\begin{rem} The inequality (ii) of Lemma   \ref{prop:ineq1}  is sharp (cf. Example \ref{ex:sym}).
Using the adjunction formula it can be rewritten as:
$$C^2\le  (p_a(C)-q)+1,$$
or, equivalently, $\rho(C)\le 1$.
\end{rem}

In the situation of Lemma \ref{prop:ineq1} (i) we can also find a lower bound for $K_SC$ using the results of Section \ref{sec:Cmoves}.
\begin{prop}\label{prop:ineq2} Let $S$ be a surface of general type with irregularity $q>1 $  that has no irrational pencil of genus $>1$  and let $C\subset S$ be a  curve with $h^0(C)=1$ and  $h^0(C|_C)=0$. Assume that   $C$ is connected and   reduced and  that every irreducible component $C_i$ of $ C$ satisfies   $C_i^2>0$; then:
$$C^2+2q-4\le K_SC,$$
or, equivalently, $C^2\le (p_a(C)-q)+1$.

Furthermore  if equality occurs then $V^1(S)$ generates $\Pic^0(S)$ and  $C$ is contained in the ramification locus of the Albanese map.

\end{prop}
 
\begin{proof} Since $h^0(C)=1$ and  $h^0(C|_C)=0$,  $0\in W(C,S)$ is an isolated point. So 
 by Theorem \ref{thm:Cmoves} (ii), $\rho(C)\leq 1$, i.e.  $q+C^2-p_a(C)\leq 1$ and this last inequality can be written $C^2+2q-4\le K_SC$.
 
The last assertion is also an immediate consequence of  Theorem \ref{thm:Cmoves} (ii).
   \end{proof}

 As  immediate  consequences of the two above propositions we obtain:
 \begin{cor}\label{cor:ineq3} Let $S$ be a surface of general type with irregularity $q>1 $  that has no irrational pencil of genus $>1$  and let $C\subset S$ be a  curve with $h^0(C)=1$. Assume that   $C$ is connected and   reduced and that   every irreducible component $C_i\subseteq C$ satisfies   $C_i^2>0$;  then:
 
$$C^2+2q-6\le K_SC,$$
or, equivalently, $C^2\le (p_a(C)-q)+2$.

Furthermore, if equality holds then $h^0(C|_C)=q$.
 \end{cor}

 \begin{cor}\label{lem:fixed}
 Let $S$ be a surface of general type with with irregularity $q>1 $  that has no irrational pencil of genus $>1$  and let $C$ be an irreducible component of the fixed part of $|K_S|$ such that $C^2>0$. Then:
$$CK_S\ge C^2+2q-4.$$
 \end{cor}
\begin{proof} We have $h^0(C)=1$ by assumption and $h^0(C|_C)=0$ by  Corollary \ref{cor:fixed-rigid}. Hence the required inequality follows by Proposition \ref{prop:ineq2}.
\end{proof}

 \bigskip 
 
In \cite{symm} we have characterized  surfaces $S$ of irregularity $q>1$ containing a curve $C$ such that $C^2>0$  and $p_a(C)=q$ (i.e., the smallest possible  value). 
By \cite{xiaoirreg} (cf. also  Corollary \ref{cor:genere}), any  irreducible curve with $h^0(C)\ge 2$ must satisfy $p_a(C)\geq 2q-1$.   
We know no example of a curve with $C^2>0$ and $q<p_a(C)<2q-1$. The next result gives some  information on this case:

\begin{cor} \label{cor:q+1}Let $S$ be a surface of general type with irregularity $q\ge 3$  that has no irregular pencil of genus $>1$  and let $C\subset S$ be an irreducible curve such that $C^2>0$ and $p_a(C)\le 2q-2$. Then:
\begin{enumerate}
\item   $C^2\le (p_a(C)-q)+1$;

\item the codimension of the tangent space at $0$ to  $W(C,S)$  is $\ge(3q-p_a(C)-3)/2\ge (q-1)/2$. \end{enumerate}

\end{cor}
\begin{proof}  Since by Corollary \ref{cor:genere}  (cf. also \cite{xiaoirreg})  we have $h^0(C)=1$, 
 by Proposition \ref{prop:BNtangent}  the tangent space to $W(C,S)$ at $0$ has dimension $w:=h^0(C|_C)$.  
 Note that by Lemma \ref{lem:hopf}  we have   $h^0(K_S|_C)+h^0(C|_C)\le p_a(C)+1<2q$. 
 
  Now observe  that $w<q$.  In fact, if $w=q$ then, by Proposition \ref{prop:xiao-petri},  one has $h^0(K_S|_C)\geq 2q-3$. Since $p_a(C)\geq h^0(K_S|_C)+h^0(C|_C)-1$ we obtain $p_a(C)\geq 3q-4$,    against the assumptions  $p_a(C)\le 2q-2$ and $q\geq 3$. So (i) follows from  Corollary \ref{cor:ineq3}.
  
 Now  Clifford's theorem gives $2w-2\le C^2$. Since $p_a(C)\leq 2q-2$, from (i) we obtain  $w\leq (p_a(C)-q+3)/2\le  (q+1)/2$. Statement (ii) then follows since $w$ is the dimension of the tangent space to $W(C,S)$ at $0$. 
\end{proof}

\subsection{The fixed part of the paracanonical system}\label{ssec:para}

Let $S$ be a smooth surface  of general type of  irregularity $q\ge 2$ such that $\albdim S=2$. Recall (cf. \cite{beauville-annulation}, \S 3)  that the {\em paracanonical system} $\{K_S\}$  of $S$ is the connected component of the Hilbert scheme of $S$  containing a canonical curve.  There is a natural morphism $c\colon \{K_S\}\to \Pic^0(S)$ defined by $[C]\mapsto \OO_S(C-K_S)$ and the fiber of $c$ over $\eta\in \Pic^0(S)$ is the linear system $|K_S+\eta|$, hence there is precisely one irreducible component ${\mathcal K}_{main}$ of $\{K_S\}$ (the so-called {\em main  paracanonical system}) that dominates $\Pic^0(S)$. By the generic vanishing theorem of Green and  Lazarsfeld, one has  $\dim |K_S+\eta|=\chi(S)-1$ for $\eta \in \Pic^0(S)$ general, and so the main paracanonical system ${\mathcal K}_{main}$  has dimension $q+\chi(S)-1=p_g(S)$. It is known (\cite[Prop.4]{beauville-annulation}) that if $q$ is even and $S$ has no irrational pencil of genus $>q/2$, then the canonical system $|K_S|$ is an irreducible component of $\{K_S\}$.

 The relationship between the fixed part of ${\mathcal K}_{main}$ and the fixed part of $\{K_S\}$ does not seem to have been studied in general. Here we relate the fixed part of ${\mathcal K}_{main}$ to the ramification locus of the Albanese map.
 
 \begin{prop}\label{prop:para}
 Let $S$ be a smooth surface  of general type of  irregularity $q\ge 2$ that has no irrational pencil of genus $>\frac q 2$ and let $C\subset S$ be an irreducible curve  with $C^2>0$ that is contained in the fixed part of the main paracanonical system ${\mathcal K}_{main}$.
 
 Then $C$ is contained in the ramification locus of the Albanese map of $S$.
 \end{prop}
 \begin{proof} By the semi-continuity of the map $\eta\mapsto h^0(K_S-C+\eta)$,  $\eta \in \Pic^0(S)$, we have $h^0(K_S-C)\ge \chi(S)$.  Assume for contradiction that $C$ is not contained in the ramification divisor of the Albanese map: then by  Corollary  \ref{cor:CC} (i) we have $h^0(K_S-C)=\chi(S)$. By Proposition \ref{prop:BNtangent} it follows that the bilinear map $H^1(\OO_S)\otimes H^0(K_S-C)\to H^1(K_S-C)$ given by cup product is zero. Hence for every section $s\in H^0(K_S)$ that vanishes along $C$  and for every $v\in H^1(\OO_S)$ we have $s\cup v=0$.
  Therefore, by the proof of Proposition \ref{prop:dv}, it follows that if  $\alpha,\beta\in H^0(\Omega^1_S)$ are such that $\alpha\wedge \beta\ne 0$, then  $\al\wedge \beta$  does not vanish along $C$. 
  
  Consider the Grassmannian $\mathbb G: = \mathbb G(2, H^0(\Omega^1_S))\subseteq \pp(\bigwedge^2H^0(\Omega^1_S))$ and the projectivization  $T\subset \pp(\bigwedge^2H^0(\Omega^1_S))$ of the kernel of $\bigwedge^2H^0(\Omega^1_S)\to H^0(K_S)$. By the Theorem \ref{thm:cast-cat}  the  intersection $T\cap \mathbb G$ is the union of a finite number of Grassmannians  $\mathbb G(2, V)\subset \pp(\bigwedge^2 V)$ where $V\subset H^0(\Omega^1_S)$ is a subspace of the form $p^*H^0(\omega_B)$ for $p\colon S\to B$ an irrational pencil of genus $>1$. Since by assumption $S$ has no irrational pencil of genus $>\frac{q}{2}$, if  $\mathbb G_0\subset \mathbb G$ is a general  codimension $q-3$ hyperplane section then  $\mathbb G_0\cap T=\emptyset$. Hence the image of $\mathbb G_0$ in $|K_S|$ is a closed subvariety $Z$ of dimension $q-1$. Hence $(C+|K_S-C|)\cap Z$ is nonempty, namely there exist $\al,\be\in H^0(\Omega^1_S)$ such that $\al\wedge \beta\ne 0$ and $\al\wedge \beta$ vanishes on $C$, a contradiction. 
 \end{proof}

\section{Examples and open questions}\label{sec:examples}
We collect here some examples to illustrate the  phenomena that one encounters in studying  the Brill-Noether loci of curves on irregular surfaces. We also give an example (Example \ref{ex:V1}) that shows that the hypothesis that $V^1(S)$ does not generate $\Pic^0(S)$ in Theorem \ref{thm:Cmoves} is not empty, i.e. that surfaces $S$ of maximal Albanese dimension without  irrational pencils of genus $>1$ such that  $V^1(S)$ generates $\Pic^0(S)$ do exist. 
We conclude  the section by  posing  some  questions.  
\begin{ex}[Symmetric products]\label{ex:sym} 
Let $C$ be a smooth curve of genus $q\ge 3$ and let $X:=S^2C$ be the second  symmetric product  of $C$. The surface $X$  is  minimal  of general type with irregularity $q$ (cf. \cite[\S 2.4]{survey} for a detailed description  of $X$). 

For any $P\in C$,  the curve    $C_P=\{P+x|x\in C\}\subset X$ is a smooth curve isomorphic to $C$, in particular  it has genus $q$. It  satisfies $C^2_P=1$, $h^0(C_P)=1$ and $\rho(C_P)=1$. 
If we fix $P_0\in C$,  then it is easy to check that the map $C\to W(C_{P_0}, X)$ defined by $P\mapsto C_P-C_{P_0}$ is an isomorphism, hence Theorem \ref{thm:Cmoves} is sharp in this case.

 Notice also that  for every $P\in C$ we have  $h^0(K_X|_{C_{P}})=q-1$, hence both Proposition \ref{prop:xiao} and Proposition \ref{prop:xiao-petri}  are sharp in this case. \end{ex}
\begin{ex}[\'Etale double covers of symmetric products]\label{ex:cover-sym}
As in Example \ref{ex:sym}, take $C$ a smooth curve of genus $q\ge 3$, let $X:=S^2C$ be the second  symmetric product and for $P\in C$ let $C_P=\{P+x|x\in C\}\subset X$. 
Let $f\colon C'\to C$ be an \'etale double cover and let $X':=S^2C'$. The involution $\si$ of $C'$ associated to the covering $C'\to C$ induces an involution $\tau$ of $X'$  defined  by $\tau(A+B)=\si(A)+\si(B)$.  The fixed locus  of $\tau$ is the smooth curve  $\Ga=\{A+\si(A)|A\in C'\}$, hence $Y:=X'/\tau$ is a smooth surface. It is easy to check that $q(Y)=q$ and that $f$ descends to a degree 2 \'etale cover $Y\to X$.

Denote by $D_P$ the inverse image of $C_P$ in $Y$. 
The map $D_P\to C_P$ is a connected  \'etale double cover, hence $D_P$ is smooth (isomorphic to $C'$)  with $D^2_P=2$ and $g(D_P)=2q-1$. In this case $\rho(D_P)=3-q\le 0$ although $D_P$ moves in a $1$-dimensional algebraic family.

Next we  study  $h^0(D_P)$. The standard restriction sequence for $D_P$ on $Y$  gives $0\to H^0(\OO_Y)\to H^0(D_P)\to H^0(D_{P}|_{D_{P}})\to H^1(\OO_Y)$ and by Proposition \ref{prop:BNtangent} the last map in the sequence is nonzero for every $P$ since $D_P$ moves algebraically. Hence if $H^0(D_P|_{D_{P}})=1$ (e.g., if $C'$ is not hyperelliptic) then $h^0(D_P)=1$.
Consider now a special case: take $C$ hyperelliptic, $A,B\in C$ two Weierstrass points and $C'\to C$ the double cover given by the $2$-torsion element $A-B$, so that the corresponding \'etale double cover $Y\to X$ is given by the equivalence relation  $2(C_A-C_B)\equiv 0$. Then  the curves $D_A$ and $D_B$ are linearly equivalent on $Y$ and we  have $h^0(D_A)= 2$. 

With a little extra work it is possible to show that this is the only instance in which $h^0(D_P)>1$.
\end{ex}

\begin{ex}[The Fano surface of the cubic threefold]                       
This  example deals with a well-known $2$-dimensional family of curves of genus $11$ on a surface
 of irregularity $q=5$. 
 
Let  $V=\{f(x)=0\}\subset \pp^4$ be a smooth  cubic $3$-fold. 
 Let  ${\mathbb G}:={\mathbb G}(2,5)$ be the Grassmannian of   lines of  $\bP^4.$
 The Fano surface (see \cite {CG} and \cite{T})  is the subset $F=F(V)\subset \bG$
of  lines contained in $V$. We recall that $F$ is a smooth surface with irregularity $q=5$.   
  In fact, let $J(V)$ be the intermediate Jacobian of $V$:   
 the Abel-Jacobi map $F\to J(V)$ induces an isomorphism 
  $\Alb(F){\rightarrow} J(V)$.
Moreover one has: 
$H^2(F,\bC)\cong \bigwedge^2 H^1(F,\bC)$, 
$K_F^2=45$, $\chi(\cO_F)=6$,  $p_g(F)=10$ and $c_2(F)=27$.

Following Fano,   for any $r\in F$ we consider the curve $C_r\subset F$
$$C_r=\{s\in F :s\cap r\neq \emptyset \}.$$
We have: 
$h^0(C_r)=1$, $C_r^2=5$,  $p_a(C_r)=11$ and 
  $K_F\sim_{alg} 3C_r$;   in addition,  the  general $C=C_r$ is smooth (see \cite{CG} and \cite{T}).

We remark that   $W(C,F)$ contains a $2$-dimensional variety isomorphic to $F$, while  
 one would expect  it  to be empty, since  $\rho(C)=-1$. On the other hand,  since the family has dimension $2$,  we have  
$h^0(\cO_C(C))=2$, hence    $W_5^1(C)\ne \emptyset$ and  $C$ is not  Brill-Noether general.  In fact the corresponding Brill-Noether number is
$\rho(11,5,1)=11-2(11-5+1)=-3$.
Moreover there is a degree $2$  \'etale map $C\to D$,  where $D$ is a smooth plane quintic, and 
$\Alb(F)$ is isomorphic to the Prym variety $P(C,D)$ of the covering, thus  the curve $C$
 has  very special moduli.
 
Nevertheless we will see that the family $\{C_r\}_{r\in F}$ has a good infinitesimal behavior.
Firstly we recall that  $h^0(K_F-C_r)=3$  by \cite[Cor. 2.2]{T}.  Let $T$ be the tangent space to $W(C_r,F)$ at $0$; since  $\dim T\ge 2$,  the image $V\subseteq H^2(\OO_F)\cong \wedge^2H^1(\OO_F)$ of $T\otimes H^1(\OO_F)$ has dimension $\ge 7$. On the other hand,  $V$ is orthogonal to the image in $H^0(K_S)$ of the Petri map $\beta_{C_r}\colon H^0(C_r)\otimes H^0(K_F-C_r)\to H^0(K_F)$, therefore  $\dim V\le p_g(F)-h^0(K_F-C_r)=7$. So we have $\dim V=7$,  $\dim T=2$ and 
 in this case the dimension of the family is predicted by the Petri map.  
\end{ex}

\begin{ex}[Ramified double covers] \label{ex:double}
Let $X$ be a smooth surface of irregularity $q$ such that the Albanese map $X\to \Alb(X)$  is an embedding  (for instance take $X$ a complete intersection in an abelian variety), and let $\pi\colon S\to X$ be the double cover given by a  relation $2L\equiv B$ with $B$ a smooth ample curve. Write $\pi^*B=2R$; the induced map $\Alb(S)\to \Alb(X)$ is an isomorphism (cf. \cite[\S 2.4]{survey}), hence $R$ is the ramification divisor of the Albanese map of $S$. We have $R\in |\pi^*L|$ and, by the projection formula for double covers, for every $\eta\in \Pic^0(S)=\Pic^0(X)$  we have $H^0(R+\eta)=H^0(L+\eta)\oplus H^0(\eta)$, where the first summand is the space of invariant sections and the second one is the space of anti-invariant sections.  Hence for $\eta\ne 0$  all sections are invariant, while  for $\eta=0$ the curve  $R$ is the zero locus of the only (up to scalars) anti-invariant section. Hence $R$ moves only linearly on $S$, as predicted by Proposition \ref{prop:dv}.
 
This construction can be used also to produce examples of surfaces of fixed irregularity $q$ that contain smooth curves $C$ with $C^2>0$, $h^0(C)=1$ and unbounded genus.  Assume that $X$ contains a  smooth curve $D$ such that $D^2>0$ and $h^0(D)=1$ (for instance,   take $X$ a symmetric product as in Example \ref{ex:sym}). Set $C:=\pi\inv(D)$: if $B$ meets $D$ transversally, then $C$ is smooth of genus $2g(D)-1+LD$, hence  $g(C)$ can be arbitrarily large. Again by the projection formulae, 
for every $\eta\in \Pic^0(S)=\Pic^0(X)$ we have $h^0(C+\eta)=h^0(D+\eta)+h^0(D+\eta-L)$.  Hence if $L-D>0$, we have   $h^0(C+\eta)=h^0(D+\eta)=1$ and $W(C,S)=W(D,X)$. 
\end{ex}

\begin{ex}\label{ex:V1}
An example of surface $S$  without pencils of genus $>1$  such that  $V^1(S)$ generates $\Pic^0(S)$ can be constructed as follows.  

Let $E$ be an elliptic curve,  let $C\to E$ and   $E'\to E$ be double covers with $C$ a curve of genus 2 and $E'$ an elliptic curve,  and set $B:=C\times_E E'$. The map $B\to E$ is a $\Z_2^2$-cover and $B$ has genus 3. We denote by $\al$ the element of the Galois group of $B\to E$  such that $B/\!\!<\al>=E'$ and by $\be,\ga$ the remaining nonzero elements. The curves $B/\!\!<\be>$ and $B/\!\!<\ga>$ have genus 2.

Now  choose  elliptic curves $E_1, E_2,E_3$ and for $i=1, 2,3$ let $B_i\to E_i$ and  $\al_i,\be_i, \ga_i$ be as above. Let $X:=B_1\times B_2\times B_3$ and   let $G$ be  the  subgroup of $\Aut(X)$ generated by $g_1=(\al_1,\be_2,\ga_3)$ and $g_2=(\be_1,\ga_2,\al_3)$;  note that $G$ acts freely on $X$.  Let $S'\subset X$ be a smooth ample divisor which is invariant under the $G$-action and let $S:=S'/G$; we denote by $f_i\colon S\to E_i$ the induced map, $i=1,2,3$. The surfaces $S'$ and $S$ are minimal of general type. 
By the  Lefschetz Theorem, $\Alb(S')=\Alb(X)=J(B_1)\times J(B_2)\times J(B_3)$. It is immediate to check that $q(S)=3$ and that  the map $S\to A= E_1\times E_2\times E_3$ induces an isogeny $\Alb(S)\to A$. 
Consider the \'etale cover $S_1=S'/\!\!<g_1>\to S$; the map $S'\to B_1$ induces a map $S_1\to E'_1=B_1/\!\!<\al_1>$  which  is equivariant for the action of $G/\!\!<g_1>$. The group  $G/\!\!<g_1>$ acts freely on $E'_1$,   hence the cover $S_1\to S$ is obtained from $E'_1\to E_1$ by base change. It follows that the   $2$-torsion element $\eta_1\in \Pic(S)$ associated to this  double cover is a  pull back from $E_1$. Thus $\eta_1$ belongs  to $\Pic^0(S)$. 

The map $S'\to B_2$ induces a fibration  $S_1\to C_2=B_2/<\be_2>$. There is a commutative diagram
\[
\begin{CD}
S_1@>>>S\\
@VVV@VV{f_2}V\\
C_2@>>>E_2,
\end{CD}
\]
where the map $S_1\to S$ is obtained from $C_2\to E_2$ by base change and normalization. This means that the fibration $S_1\to C_2$ has two double fibres $2F_1$ and $2F_2$, occurring at the ramification points of $f_2$  and that $\eta_1=F_1-F_2+\alpha$ for some $\alpha\in \Pic^0(E_2)$, hence  $\eta_1$ restricts to $0$ on the general fiber of $f_2$. 
By \cite[Thm.2.2]{beauville-bayreuth}, this implies that $\eta_1+f_2^*\Pic^0(E_2)$ is a component of $V^1(S)$.
A similar argument shows that $V^1(S)$ contains translates of $f_1^*\Pic^0(E_1)$ and $f_3^*\Pic^0(E_3)$. Since $\Alb(S)$ is isogenous to $E_1\times E_2\times E_3$, it follows that $V^1(S)$ generates $\Pic^0(S)$.

To conclude, we show that if
 the curves $E_i$ are general, then $S$ has no irrational pencil of genus $>1$.
 In this case $\Hom(E_i,E_j)=0$ if $i\ne j$, hence $\End(A)=\Z^3$. Assume for contradiction that $S\to B$ is an irrational pencil of genus $b>1$. Since the map $S\to A$ is generically finite by construction and $q(S)=3$, $b=2$ is the only possibility. 
  Then we have a map with finite kernel  $J(B)\to A$. Let $W$ be the image of $J(B)$ in $A$. Consider the endomorphism $\phi$ of $A$ defined as  $ A\to A/W=(A/W)^{\vee}\to  A^{\vee}=A$ (both $A$ and $A/W$ are principally polarized). The connected component of $0\in \ker \phi$ is $W$, hence $W$ is a product  of two of the  $E_i$. So the map $S\to W$  has finite fibres, while $S\to J(B)$ is composed with a pencil, and  we have a contradiction.
 \end{ex}
 \begin{qst}\label{qst:Q1} 
 Let $S$ be a surface of general type  of irregularity $q$ and of maximal Albanese dimension.  A curve $C\subset S$ with $C^2>0$ satisfies $p_a(C)\ge q$ and in \cite{symm} we have proven   that if $S$ contains a  $1$-connected curve $C$ with $C^2>0$ and $p_a(C)=q$, then  $S$  is birationally a product of curves or the  symmetric square of a curve. 
 On the other hand, the curves $D_P$ in Example \ref{ex:cover-sym} have $D^2_P>0$ and genus equal to $2q-1$.  We do not know any surface  $S$ containing a curve $C$ with $C^2>0$ and   arithmetic genus in the intermediate range $q<p_a(C)<2q-1$,  so it is natural to ask whether such an example exists. Notice that,  by   Corollary \ref{cor:genere} (cf. also \cite{xiaoirreg}),   a curve  $C$ with  $C^2> 0$ and $p_a(C)<2q-1$ cannot move linearly and that some  further restrictions  are given in Corollary \ref{cor:q+1}.   In addition, the image $C'$ of  such a curve $C$ via the Albanese map  generates $\Alb(C)$ and therefore by the Hurwitz formula $C'$  is birational to $C$. Hence this question is also related to the question of existence of curves of genus $q< p_a(C)<2q-1$ that generate an abelian variety of dimension $q$ (see \cite{pi}   for related questions).  
  \end{qst}

 \begin{qst} On a variety  $X$ with $\albdim X=\dim X $ there are three intrinsically defined effective divisors:

 \begin{itemize}
 \item[(a)] the fixed part of $|K_X|$;
 \item[(b)] the ramification divisor $R$ of the Albanese map;
 \item[(c)] the fixed part of the main paracanonical system ${\mathcal K}_{main}$ (cf. \S \ref{ssec:para}).
 \end{itemize}
 Clearly the fixed part of $|K_X|$ is a subdivisor of  $R$.  In the case of surfaces,  in Proposition \ref{prop:para} it is shown that the components $C$ of the fixed part of ${\mathcal K}_{main}$ with $C^2>0$ are contained in $R$ if the surface has no irrational pencil of genus $>\frac q 2$. 
 
  It would be interesting to know more precisely  in arbitrary dimension how these three divisors are related.\footnote{In our recent preprint \cite{para}, we have proven by the different methods that for surfaces without irrational pencils of genus $>\frac q 2$ the fixed part of the main paracanonical system is contained in the fixed part of the canonical system.}
   \end{qst}
  \begin{qst} 
In  \S \ref{sec:Cmoves} we study the Brill-Noether loci for \underline{curves} $C$ on a surface $S$, namely we always assume that $0\in W(C,S)\ne \emptyset$. It would be very interesting to find numerical conditions on a line bundle $L\in\Pic(X)$,  $X$  a smooth projective variety, that ensure that $W(L,X)$ is not empty.

\end{qst}


\bigskip

\begin{minipage}{13.0cm}
\parbox[t]{6.5cm}{Margarida Mendes Lopes\\
Departamento de  Matem\'atica\\
Instituto Superior T\'ecnico\\
Universidade T{\'e}cnica de Lisboa\\
Av.~Rovisco Pais\\
1049-001 Lisboa, PORTUGAL\\
mmlopes@math.ist.utl.pt
 } \hfill
\parbox[t]{5.5cm}{Rita Pardini\\
Dipartimento di Matematica\\
Universit\`a di Pisa\\
Largo B. Pontecorvo, 5\\
56127 Pisa, Italy\\
pardini@dm.unipi.it}

\vskip1.0truecm

\parbox[t]{5.5cm}{Gian Pietro Pirola\\
Dipartimento di Matematica\\
Universit\`a di Pavia\\
Via Ferrata, 1 \\
 27100 Pavia, Italy\\
\email{gianpietro.pirola@unipv.it}}
\end{minipage}

\end{document}